\documentclass[epic,11pt]{article}
\usepackage{amsfonts,amsmath,amscd,amssymb,amsthm}
\usepackage{fullpage}
\usepackage[alphabetic,initials]{amsrefs}
\usepackage[all]{xy}
\usepackage{hyperref} 
\usepackage{mathrsfs}
\input xy%
\usepackage{fullpage}
\def\convf{\hbox{\space \raise-2mm\hbox{$\textstyle      \bigotimes \atop \scriptstyle \omega$} \space}}
\def\0{{\bar 0}}
\def\1{{\bar 1}}

\def\c{{\mathbb C}}
\def\Z{{\mathbb Z}}

\def\Max{{\operatorname{Max}\;}}

\def\mx {{\operatorname{max}}}

\def\Spec  {{\operatorname{Spec  }\;}}

\newcommand{\ttk}{\mathtt k}
\newcommand{\Xq}{X/\!\!/\sim}

\newcommand{\fG}{\mathfrak{G}}
\newcommand{\itemiiip}{\item[{{\rm(iii´)}}]}
\newcommand{\itemi}{\item[{{\rm(i)}}]}
\newcommand{\itemii}{\item[{{\rm(ii)}}]}
\newcommand{\itemiii}{\item[{{\rm(iii)}}]}
\newcommand{\itemiv}{\item[{{\rm(iv)}}]}

\newcommand{\itema}{\item[{{\rm$($a$)$}}]}
\newcommand{\itemb}{\item[{{\rm$($b$)$}}]}
\newcommand{\itemc}{\item[{{\rm$($c$)$}}]}
\newcommand{\itemd}{\item[{{\rm$($d$)$}}]}

\newcommand{\noi}{\noindent}

\newcommand{\Gl}{\Lambda}
\newcommand{\Gd}{\Delta}

\newcommand{\gs}{\sigma}

\newcommand{\gt}{\tau}

\newcommand{\gl}{\lambda}

\newcommand{\gk}{\kappa}

\newcommand{\gth}{\theta}

\newcommand{\ot}{\otimes}

\newcommand{\fg}{\mathfrak{g}}\newcommand{\fgl}{\mathfrak{gl}}

\newcommand{\fh}{\mathfrak{h}}

\newfont{\eufm}{eufm10 scaled\magstep1}

\newcommand{\bcu}{\bigcup}

\newcommand{\bsk}{\backslash}

\newcommand{\cO}{\mathcal{O}}
\newcommand{\cC}{\mathcal{C}}

\newcommand{\cV}{\mathcal{V}}

\newcommand{\cW}{\mathfrak{W}}

\newcommand{\ey}{\end{eqnarray}}
\newcommand{\by}{\begin{eqnarray}}

\newcommand{\bco}{\begin{conjecture}}
\newcommand{\eco}{\end{conjecture}}
\newcommand{\bq}{\begin{question}}
\newcommand{\eq}{\end{question}}
\newcommand{\ba}{\begin{alg}}
\newcommand{\ea}{\end{alg}}
\newcommand{\bpf}{\begin{proof}}
\newcommand{\epf}{\end{proof}}
\newcommand{\bt}{\begin{theorem}}

\newcommand{\et}{\end{theorem}}

\newcommand{\br}{\begin{rem}}
\newcommand{\er}{\end{rem}}
\newcommand{\brs}{\begin{rems}}
\newcommand{\ers}{\end{rems}}
\newcommand{\bi}{\begin{itemize}}
\newcommand{\ei}{\end{itemize}}
\newcommand{\bl}{\begin{lemma}}
\newcommand{\bsul}{\begin{sublemma}}
\newcommand{\esul}{\end{sublemma}}
\newcommand{\bp}{\begin{proposition}}
\newcommand{\be}{\begin{equation}}
\newcommand{\bc}{\begin{corollary}}
\newcommand{\bexs}{\begin{examples}}
\newcommand{\eexs}{\end{examples}}
\newcommand{\bexa}{\begin{example}}
\newcommand{\eexa}{\end{example}}
\newcommand{\bex}{\begin{exercise}}
\newcommand{\eex}{\end{exercise}}
\newcommand{\btab}{\begin{tab}}
\newcommand{\etab}{\end{tab}}
\newcommand{\el}{\end{lemma}}
\newcommand{\ep}{\end{proposition}}
\newcommand{\ee}{\end{equation}}
\newcommand{\ec}{\end{corollary}}
\newcommand{\Bc}{\begin{center}}
\newcommand{\Ec}{\end{center}}
\newcommand{\bh}{\begin{hyp}}
\newcommand{\eh}{\end{hyp}}
\newcommand{\bhs}{\begin{hyps}}
\newcommand{\ehs}{\end{hyps}}
\newcommand{\bd}{\begin{dfn}}
\newcommand{\ed}{\end{dfn}}

\newcommand{\bn}{\begin{notn}}
\newcommand{\en}{\end{notn}}

\begin{document}
\title{Table of Contents}

\newtheorem{thm}{Theorem}[section]
\newtheorem{hyp}[thm]{Hypothesis}
 \newtheorem{hyps}[thm]{Hypotheses}
\newtheorem{notn}[thm]{Notation}

  \newtheorem{rems}[thm]{Remarks}
\newtheorem{questions}[thm]{Questions}
\newtheorem{question}[thm]{Question}
\newtheorem{conjecture}[thm]{Conjecture}
\newtheorem{theorem}[thm]{Theorem}
\newtheorem{theorem a}[thm]{Theorem A}
\newtheorem{example}[thm]{Example}
\newtheorem{examples}[thm]{Examples}
\newtheorem{corollary}[thm]{Corollary}
\newtheorem{rem}[thm]{Remark}
\newtheorem{lemma}[thm]{Lemma}
\newtheorem{sublemma}[thm]{Sublemma}
\newtheorem{cor}[thm]{Corollary}
\newtheorem{proposition}[thm]{Proposition}
\newtheorem{exs}[thm]{Examples}
\newtheorem{ex}[thm]{Example}
\newtheorem{exercise}[thm]{Exercise}
\numberwithin{equation}{section}%
\setcounter{part}{0}

\newcommand{\drar}{\rightarrow}
\newcommand{\lra}{\longrightarrow}
\newcommand{\rra}{\longleftarrow}
\newcommand{\dra}{\Rightarrow}
\newcommand{\dla}{\Leftarrow}
\newcommand{\rl}{\longleftrightarrow}
\newtheorem{Thm}{Main Theorem}
\newtheorem*{thm*}{Theorem}
\newtheorem{lem}[thm]{Lemma}
\newtheorem*{lem*}{Lemma}
\newtheorem*{prop*}{Proposition}
\newtheorem*{cor*}{Corollary}
\newtheorem{dfn}[thm]{Definition}
\newtheorem*{defn*}{Definition}
\newtheorem{notadefn}[thm]{Notation and Definition}
\newtheorem*{notadefn*}{Notation and Definition}
\newtheorem{nota}[thm]{Notation}
\newtheorem*{nota*}{Notation}
\newtheorem{note}[thm]{Remark}
\newtheorem*{note*}{Remark}
\newtheorem*{notes*}{Remarks}
\newtheorem{hypo}[thm]{Hypothesis}
\newtheorem*{ex*}{Example}
\newtheorem{probs}[thm]{Problems}
\newtheorem{conj}[thm]{Conjecture}
\newtheorem{quests}[thm]{Questions}
\newtheorem{quest}[thm]{Question}
\newtheorem{prob}[thm]{Problem}

\title{ Quotients for equivalence relations on affine varieties}
\author{Ian M. Musson\\Department of Mathematical Sciences\\
University of Wisconsin-Milwaukee\\ email: {\tt
musson@uwm.edu}}
\maketitle

\begin{abstract}  We consider an affine variety $X$ with an equivalence relation $\sim$ on the set of closed points $X(\ttk)$.  Let $R=\cO(X)^\sim \subset \cO(X) =S$ 
be the ring of regular functions on $X$ which are constant on equivalence classes. We adapt the definition of a geometric quotient from 
  \cite{MFK} to this situation.  The main result asserts that if 
$\pi: X \lra Y=  \Spec \cO(X)^\sim$ is a geometric quotient and $Y$ is a Jacobson scheme, then   $\pi$
is universal in  the  category of $\ttk$-schemes ($\pi$
is a categorical quotient).
By a result of Amitsur \cite{Amit}, the Jacobson condition holds if $\ttk$
 is uncountable.  One condition we require is that $\pi$ is surjective on closed points.  Using the terminology of \cite{M24}, this means that  the pair $(R, S)$ satisfies the relative weak Nullstellensatz (RWN).

Examples of pairs $(R, S)$ with RWN arise in the representation theory of a  Lie superalgebra $\fg$, \cite{GHS}, \cite{M24}.  In the latter case we use the main theorem to show that $\pi $ is a geometric and categorical quotient, provided  $\ttk$ is uncountable.
Here the equivalence classes are orbits of  the Weyl groupoid $\cW$ from   \cite{SV2} on  $X(\ttk)$ and $X$ is chosen so that  the invariant ring  $\cO(X)^\cW $ is related either to the center 
$Z(\fg)$ or the Grothendieck ring of $\fg$. The examples  in \cite{GHS} involve  semiregular Lagrangian equivalence relations and it is shown in \cite{GHS} Theorem 2.10 that 
 $(R, S)$ satisfies RWN.
\end{abstract}
  

\section{Introduction}\label{intro}  Throughout  $\ttk$ denotes  an algebraically closed field and iff means if and only if.  
For a $\ttk$-algebra $R$, let $\Max  R$ denote the space of maximal ideals of $R$ with the subspace topology inherited from $\Spec R$. 
Consider an  extension  $R \subseteq S$ of $\ttk$-algebras with $S$ finitely generated.  
If 
$M \in \Max S $, then $S/M \cong \ttk$, so if  $m=M\cap R$, then \be 
\label{sgs4}  \ttk \hookrightarrow R/m \cong (R +M)/M \subseteq S/M \cong\ttk. \ee  Thus $m \in \Max  R$ and the map $\pi:X=\Spec S\lra Y=\Spec R$ restricts to a map 
 \be \label{sgs51}\pi_\mx: \Max S \lra  \Max R, \quad M \mapsto M\cap R . \ee  
As in \cite{M24}, we say the {\it relative weak Nullstellensatz} (RWN) for the pair $(R, S)$ if  $\pi_\mx$ is onto.  By Zorn's Lemma an equivalent condition is:  if $I\neq R$ is an ideal of $R$, then  $IS \neq S$. 
If RWN  holds for the pair $(R, S)$, and $m \in \Max R$, then  for any $M \in \Max S $ containing $mS$, we have $M\cap R = m$.  Thus we can identify 
$ \Max R$ with the set of equivalence classes in $\Max S$  for the equivalence relation determined by the fibers of $\pi_\mx$. 
  Since the residue fields are $\ttk$ in both cases, maximal ideals of $R, S$ correspond to $\ttk$-points  so 
we have identifications
\be \label{eq2} X(\ttk)=\Max S, \quad Y(\ttk)=\Max R.\ee
Recall that   a ring is {\it Jacobson} if any prime ideal is an intersection of maximal ideals.  
To say anything about $\pi$ it is natural to assume that $R$  is Jacobson, since then as observed in \cite{DuK}, the Zariski topology  on $Y$ is determined by its restriction to the set of $\ttk$-points, see Subsection \ref{JC}.
\\ \\
A case of particular interest arises for invariant rings. 
For actions  affine algebraic groups there is a vast literature, 
\cite{BB}, \cite{DK}, \cite{K84}, \cite{KSS}, \cite{MFK}. 
However it is widely recognized that symmetries of geometric structures are often best understood using groupoids, and for the  action of  a groupoid on $X$, there is little general theory. More generally, suppose 
$\sim $
is an equivalence relation  on the $\ttk$-points of an  affine variety $X$. 
The invariant ring 
$R=\cO(X)^\sim $ consists of all functions in $S=\cO(X) $ that are  constant 
on equivalence classes in 
$ X(\ttk)$.
\\ \\
In \cite{MFK} Proposition 0.1 it is shown that for group actions, a geometric quotient is universal in the 
category of $\ttk$-schemes.  Our first main result Theorem \ref{prop3}, gives an analog of this result for quotients  by  equivalence relations,  provided the quotient is a Jacobson scheme.  We note that quotients by equivalence relations are studied in \cite{BB} and \cite{Ko}. 
\\ \\The definitions  in \cite{MFK}  are formulated for an action of a group pre-scheme over a base.  In particular \cite{MFK} Definition 0.3 gives  axioms for an  action using commutative diagrams.  This does not seem appropriate for actions of groupoids.  Instead, 
for a set  $X$ 
let $\mathfrak A(X)$ be the  groupoid with  base all  subsets of $X$ and  morphisms all 
bijections between subsets. The groupoid $\fG$ acts on   $X$   if there is a functor 
\be \label{bwr}\gt:\fG \lra \mathfrak A(X).\ee
Consider the equivalence relation $\sim$ on $X(\ttk)$ which is generated by the symmetric relation
$x \sim y$  iff there is a morphism  $f: A\lra B$ in $\fG$ such that $x \in \gt(A)$ and $y = \gt(f)(x)$.  The equivalence classes under $\sim$ are called    
 {\it $\fG$-orbits}. We denote $\cO(X)^\sim $ by $\cO(X)^\fG$ in this setting. Thus \[\cO(X)^\fG  =\{f \in \cO(X)| f \mbox{  is constant on } \fG \mbox{-orbits in } X(\ttk)\}. \] 
If $R=S^\fG$ is the fixed ring for certain groupoid actions, the map $\pi: X =  \Spec S \lra Y=  \Spec R$ has some desirable geometric properties,   \cite{M24}. Here $S$ is either a polynomial algebra or a Laurent polynomial algebra in finitely many variables, and $\fG$ is closely related to the Weyl groupoid 
introduced in  \cite{SV2}.  
We use the main result to show that for $\ttk$ is uncountable,  $\pi$ is a geometric and  categorical quotient. 

 \section{Geometric and Categorical Quotients}
Let us recall some definitions. Throughout 
$\ttk$ is an algebraically closed field and a 
{\it variety $X$} is an integral quasi-projective $\ttk$-scheme of finite type, 
\cite{Ha} Proposition II.4.10.  Each closed point of $X$ has residue field isomorphic to $\ttk$.  We denote the set of such points by $X(\ttk)$. 
Note that for every open subset $V$ of $X$, the algebra 
${\mathcal {O}}_{X}(V)$ can be identified with the algebra of $\ttk$-valued functions on $V(\ttk)$.  This is because the 
set of closed points of $X$ is dense by \cite{Ha} Exercise II.3.14.  Suppose that $\sim $ is an equivalence relation  on $X(\ttk)$.  
\noi  
Let $Y$ be a $\ttk$-scheme.  
 The   morphism $\pi:X \lra Y$    is 
{\it universal } in the category $\cC$ if    
\bi \itema $\pi$ is constant on equivalence classes of closed points. 
\itemb $\pi$ is universal with respect to the property in (a).  That is if 
 $\mu:X \lra Z$ is  a  morphism  in $\cC$ which
 is constant on equivalence classes of closed points, then $\mu$ factors uniquely through $\pi$.\ei
If $\pi$    is 
 universal  in the category of  $\ttk$-schemes, we say that $\pi:X \lra Y$    is  a categorical  quotient of $X$ by $\sim  $, compare \cite{MFK} Definition 0.5.
\\ \\
Suppose that $\pi$ is as in (a) above and that $U$ is open in $Y$.  If 
$V=\pi^{-1}  (U)$, then $V(\ttk)$ is a union of equivalence classes.  
We denote by 
${\mathcal {O}}_{X}(\pi^{-1}  (U))^\sim$ the functions in 
${\mathcal {O}}_{X}(\pi^{-1}  (U))$ 
which are constant on equivalence classes. The map $\pi_U^*:\cO_Y(U) \lra {\mathcal {O}}_{X}(\pi^{-1}  (U))^\sim$ given by $\pi_U^*(f) =f\pi$ is injective.
\\ \\
Each part of  the following definition is an analog of the corresponding part of \cite{MFK} Defintion 0.6.  
\bd \label{gps}A morphism of $\ttk$-schemes $\pi :X \lra Y$ is a {\it geometric quotient} if the following conditions hold 
\bi\itemi The  induced map $\pi(\ttk):X(\ttk) \lra Y(\ttk)$ is constant on equivalence classes.
\itemii $\pi(\ttk)$ is onto  and each fibre of $\pi(\ttk)$ 
is an equivalence class.
\itemiii If $A\subseteq X(\ttk)$ is a $\sim $-invariant closed  subset, then 
$\pi(A) \subseteq Y(\ttk)$ is closed. 
\itemiv 
If $U$ is any open set in $Y$ then $\pi_U^*$ is an isomorphism.
\ei Note that if $\pi(\ttk)$ is onto, then (iii') is equivalent to 
\bi
\itemiiip If $B\subseteq X(\ttk)$ is a $\sim $-invariant open subset, then 
$\pi(B) \subseteq Y(\ttk)$ is open.
\ei
 \ed \noi
\subsection{The affine case} \label{gen1} 
Suppose $X$ is affine and identify $X$ with $\Spec \cO(X)$, the set of prime ideals in $\cO(X)$ with the Zariski topology.  
If $g\in\cO(X)^\sim$, the  sets $$X_g =\{x\in X|g\notin x\} = \Spec \cO(X)_g$$ and $$Y_g =\{y\in Y|g \notin y \} = \Spec \cO(X)^\sim_g$$ are open in $X$ and $Y$ respectively, and $\pi(X_g) = Y_g$. 

\bl \label{gct} If $X$ is affine the morphism $\pi: X \lra \Spec \cO(X)^{{\sim}} =Y$ is  universal  in the category of affine schemes.\el
\bpf  It is easy to adapt the proof of \cite{K84} II.3.2.  \epf

\bl \label{jrk} If $X$ is affine, the  morphism $\pi:X \lra \Spec \cO(X)^\sim  =X/\!\!/ \sim $ satisfies  condition $($iv$)$ in Definition \ref{gps}.
\el
\bpf We have to show that for every open subset $U$ of $Y$, 
\be \label{irk}
{({\mathcal {O}}_{X}(\pi^{-1}U))^{{\sim} }}={\mathcal {O}}_{Y}(U).\ee
First assume that $U= Y_g$ is 
basic affine. 
In this case \eqref{irk} becomes 
$(\cO(X)_g)^\sim  = \cO(X)^\sim _g$ which is obvious.
In general write 
$$ U  =\bcu Y_g, \mbox{ so } \pi^{-1} U  =\bcu \pi^{-1}  Y_g.$$
Then 
$$\cO_Y(U)= \lim_{\rra} \cO_Y(  Y_g)= \lim_{\rra}\cO_X(\pi^{-1} Y_g)^\sim  
=\cO_X(\pi^{-1}  U))^\sim.$$ \epf
\noi 
\subsection{The Jacobson condition} \label{JC}
By a result of Amitsur \cite{Amit},  if $\ttk$ is an uncountable field and $R$ is   a $\ttk$-algebra of countable dimension, then $R$ is  {Jacobson}.
A scheme $Y  $ is {\it Jacobson} if for every 
closed set, $A$ we have $A = \overline{A\cap Y(\ttk)}.$
This is a local property: If $Y = \bigcup_{i\in I} Y_i$ is an open covering, then $Y$ is {Jacobson} if and only if every $Y_i$ is {Jacobson}. 
If a $\ttk$-scheme $Y$ is {Jacobson}, then the topology is determined by the induced topology on the $\ttk$-rational points $Y(\ttk)$, see Lemmas \ref{le1} and  \ref{le81}. 
\begin{lem} \label{le7} 
If $\ttk$ is uncountable and $\pi\colon X \to Y$ is a geometric quotient, then $Y$ is {Jacobson}.
\end{lem}
\begin{proof}If $U $ is open in $Y$ , then by (iv), $\cO_Y(U) \subseteq \cO_X(\pi^{-1}(U))$ which is a  $\ttk$-algebra of countable dimension, hence a {Jacobson} algebra.  Thus the result follows since the  {Jacobson} property is local.
\end{proof}
\noi For the remainder of this Subsection, 
let $R$ be a Jacobson $\ttk$-algebra,  
such that $R/m \cong \ttk$ for all   $m \in \Max  R$.
\noi 
\bl \label{le1} Let $I, J$  be radical ideals of $R$. 
\bi \itema 
Suppose that  for every maximal ideal $M$ of $R$ such that  $I \subseteq M$ we have $J \subseteq M$, then $J\subseteq I.$
\itemb If $A, B$ are  closed in $Y=\Spec R$ and ${A\cap Y(\ttk)}\subseteq {B\cap Y(\ttk)}$, then $A \subseteq B$.
\ei
\el 
\bpf It is useful to index certain sets of maximal ideals.  Thus let $$\{M\in \Max R|I \subseteq M\} =\{M_\gl| \gl\in \Gl\} \mbox{ and  
} \{M\in \Max R|J \subseteq M\}=\{M_\gth| \gth \in \Theta\}.$$
By hypothesis  
$\Gl\subseteq \Theta$  and this proves (a) since $R$ is Jacobson.
If  $A=\cV(I)$ and $B=\cV(J)$, with $I, J$ radical ideals, then  the hypothesis in (a) holds so (b) follows.
\epf
\bl \label{le81} 
If  $Y= \Spec R$, then 
\bi \itema $Y(\ttk)$ is dense in $Y$.
\itemb If $R$ is a domain and  $U$ is a non-empty open set in $Y$, then $U(\ttk) = U\cap Y(\ttk)$ is dense in $U$.  
\itemc $ Y$ is Jacobson.
\itemd If $V$ is an irreducible closed subset of $Y$, then $V  \cap  Y(\ttk)$ is irreducible
\ei \el
\bpf For (a) we  show that  any open set of the form $Y_f = \{P\in Y|f  \notin P\} $, with $f\in R$ meets $Y(\ttk)$.  We can assume $f$ is not a unit. By  a result of Krull, \cite{Kap} Theorem 1, if $P$ is maximal such that $P\cap\{f^n|n>0\}=\emptyset$, then $P$ is prime.  That $P$ exists follows from Zorn's Lemma. If $P$ is not maximal we have $f\in M$ for all  
$M\in \overline{\{P\}}$,  $M\neq P$. This cannot happen for $R$  Jacobson.
Thus 
${P}
\in Y(\ttk)\cap Y_f$.
 Under the hypothesis in (b) every non-empty open set in $Y$ is dense. Next suppose $A$ is closed in $Y$ and set $B={\overline{A\cap Y(\ttk)}}$. 
Since ${A\cap Y(\ttk)}\subseteq {B\cap Y(\ttk)}$ 
we have $A \subseteq B$ by Lemma \ref{le1}.
On the other hand  
${A\cap Y(\ttk)}\subseteq A,$  implies $B \subseteq A$ because $A$ is closed. This shows (c) and (d) follows since $V={\overline{V\cap Y(\ttk)}}$.
\epf
\subsection{Categorical Quotients}
We have the following analog of \cite{MFK} Proposition 0.1.
\bt\label{prop3}
Let $\pi\colon X \to Y$ be a geometric quotient and assume that $Y$ is {Jacobson} scheme. 
Then $\pi$ is  
a { categorical  quotient}.  
\et
\bpf
Let $\mu\colon X \to Z$ be a morphism which is constant on equivalence classes of closed points, and let $V \subset Z$ be an affine open set. Then $W:=\mu^{-1}(V)$ is open and $\sim$-invariant and so $\pi(W) \subset Y$ is open, by (iii') in Defnition \ref{gps}. This gives the following diagram of homomorphisms:
$$
\begin{CD}
\cO_Y(\pi(W)) @>{\simeq}>{\pi^*}> \cO_X(W)^\sim\\
&&@AA{\mu^*}A \\
&&\cO_Z(V)
\end{CD}
$$  Note that $\pi^*$ is an isomorphism by condition (iv) in Definition \ref{gps}.
Thus there is a uniquely defined homomorphism $\gs_V^* \colon \cO_Z(V) \to \cO_Y(\pi(W))$ 
such that $\pi^*\circ\gs_V^*  = \mu^*$. 
We give $\pi(W), V$ the induced subscheme structure coming from $Y, Z$ respectively 
\cite{Ha} Exercise II.2.2.  Then  ${\cO_Y}({\pi(W)}) = \cO(\pi(W))$ and ${\cO_Z}({V}) = \cO(V)$,  \cite{Ha} page 85 and we have 
a map 
$ \cO(V)  \lra \cO(\pi(W))$.  Since $V$ is affine, this map 
corresponds to a map $\gs_V\colon \pi(W) \lra V $ by 
 \cite{Ha} Exercise II.2.4 or \cite{EH} Theorem I-40 and we have 
$\gs_V\circ\pi|_W= \mu|_W$. Using an affine open covering $\{V_i\}_{i\in I}$ of $Z$, the different $\gs_{V_i}$ coincide on the intersections $\pi(W_i)\cap \pi(W_j) = \pi(W_i\cap W_j)$ where $W_i := \mu^{-1}(V_i)$, and thus define a morphism $\gs\colon Y \to Z$ such that $\gs\circ\pi= \mu$.
\epf 
 

\section{Weyl Groupoids} \label{awg}
From now on we assume $\ttk$ has characteristic zero. In \cite{SV2} Sergeev and Veselov introduced a groupoid $\mathfrak{W}=\mathfrak{W}(\Gd)$ in connection with the representation theory of 
a Kac-Moody Lie superalgebra $\fg$ of finite  type with set of roots $\Gd$.
Let $\fh$  be a Cartan subalgebra of 
$\fg_0$ and 
$ \mathbb{T}$ a maximal torus in a suitable  Lie  supergroup $G$ with  Lie $G =\fg$. 
Also, let $Z(\fg)$  be the center of the enveloping algebra $U(\fg)$, and $J(G)$ 
the supercharacter $\Z$-algebra of finite dimensional representations of $G$.
 There is an  action of $\cW$ on $\fh^*$ 
such that the invariant ring is  isomorphic to $Z(\fg)$. In \cite{M24}  we introduced two related  continuous   Weyl groupoids $\cW^c$  and $\cW_*^c$ 
and  actions on $X=\fh^*$ and  $ \mathbb{T}$   
such that the invariant rings 
are isomorphic to $Z(\fg)$ and  $J(G) \ot_{\Z}\ttk$ respectively.  Furthermore all  orbits 
are closed under these actions,  \cite{M24} Theorem 5.7 and Proposition 5.9. From now on, let $\fG$ denote either of the groupoids $\cW^c$  or  $\cW_*^c$ with action   on $X$ as in \cite{M24} Notation 1.1.  The strong Nullstellensatz gives a bijection between radical ideals in $\cO(X)^\fG$ and their zero loci in $X$, which we call {\it superalgebraic sets}. We need the following characterization of these sets from \cite{M24} Corollary 6.10.  
\bp\label{newp}
The superalgebraic sets in $X$  are exactly the Zariski closed sets whose $\ttk$-points  are invariant under the Weyl groupoid.
\ep
\bt \label{newcor}Assume that $\ttk$ is an uncountable algebraically closed field. Then the 
morphism $\pi:X\lra X/\!\!/ \fG = \Spec \cO(X)^\fG$ is a geometric quotient and universal  in the category of $\ttk$-schemes.
\et

\bpf By Theorem \ref{prop3} it suffices to show $\pi$ is a geometric quotient.%
We check the conditions of Definition \ref{gps} hold.  The map $\pi $ is constant on $\fG$-orbits of closed points 
 and is onto by the weak Nullstellensatz,  \cite{M24} Equation (2.4). Also the fibers over a closed point are  $\fG$-orbits by \cite{M24} Corollary 5.8.  Thus conditions (i) and (ii)  hold.  Since $X$ is affine, condition (iv) holds by Lemma \ref{jrk}.
Condition (iii)  is checked in the Lemma below.
\epf 
\noi

\bl \label{rgs2}Consider the  morphism $\pi:X \lra \Spec \cO(X)^\fG =X/\!\!/ \fG$.  If $A\subseteq X(\ttk)$ is a $\fG$-invariant closed  subset, then 
$\pi(A) \subseteq (X/\!\!/ \fG)(\ttk)$ is closed. Thus $\pi$ satisfies condition  $($iii$)$ in Definition \ref{gps}.
 \el
\bpf Suppose $A = V\cap X(\ttk)$ with $V$    closed in $X$ and $A$ $\fG$-invariant.  
By Proposition \ref{newp}, we have 
  $\cV(I) =V$ 
for some  ideal  
$ I$     of $\cO(X)^\fG$.
Let $U = X \bsk V$ and $B =U \cap X(\ttk)= X(\ttk)\bsk A$.
 We claim that if $x\in B$, then for some   $g \in I,$  we have 
  $x\in X_g(\ttk) \subseteq B$. Let $m_x = \{f\in\cO(X)^\fG |f(x)= 0\}$. If $I \subseteq m_x$, we have $x\in \cV(m_x)    \subseteq \cV(I) =V$, a contradiction.  The claim holds for any $g\in I \bsk m_x$.  Now 
  we can write $B$ as a union $ B  =\bcu_{g: g \in I} X_g(\ttk)$,  \mbox{ then } $\pi(B)  =\bcu_{g: g \in I} Y_g(\ttk)$ is open in $Y$, as required. 
\epf

\subsection{Outlook} \label{qn}
We might ask whether Theorem \ref{newcor} holds without the assumption that 
$\ttk$ is uncountable.  This would be the case if we had a positive answer to the question  below.  
\bq \label{qn1} {\rm 
Suppose that $S$ is a polynomial ring  or Laurent polynomial ring in finitely many variables over $\ttk$ and $R$ is a 
subalgebra of $S$. Is $R$ a Jacobson ring?}
\eq \noi 
\subsubsection{Algebra extensions with RWN} \label{qn}
All the geometric results is obtained in \cite{M24} depend on the weak Nullstellensatz for $R$, \cite{M24} Theorem 4.1.  In \cite{GHS} the authors study equivalence relations on a finite dimensional vector space $V$.  They show that if $\sim$ is a semiregular Lagrangian equivalence relation, then the extemsion $\c[V]^\sim \subset \c[V]$  satisfies RWN. It would be interesting to know if the corresponding quotient satisfies the conditions in Definition \ref{gps}. 

 In \cite{M24}, Section 2 there is an  example of a subalgebra $R$ of the 
 polynomial algebra $S$  in two variables such that the pair $(R,S) $ does not satisfy the 
RWN.
\subsubsection{Some other non-Noetherian algebras of interest} \label{qn}
In the literature, there are several other 
examples of a non-Noetherian subalgebra $R$ of a polynomial algebra $S$ with interesting combinatorial or homological properties.  Most arise as algebras of integrals of (quantum) integrable systems or algebras generated by generalized power
   sums, \cite{ER}, \cite{SV1}, \cite{SV09}, \cite{SV101}. If the extension $R\subset S$ satisfies RWN then it should be possible to study the geometry of $\Spec R.$

In  \cite{SV101}, Sergeev and Veselov study orbits and invariants  of the Weyl groupoid $\cW$ from \cite{SV2}  associated to $\fgl(n|m)$, for an action    on $X=\ttk^{n|m}$. This action  depends on a non-zero parameter $\gk.$ 
By \cite{SV101} Theorem 1.1,  the fixed algebra is non-Noetherian iff $\gk=-p/q$ where $1\le p\le m$  and $1\le q\le n$.  If $\gk=-1$, the action is to that considered in \cite{M24} where the fixed algebra is isomorphic to $Z(\fg)$ and RWN holds.

\end{document}